\documentclass[11pt]{article}
\usepackage{amsfonts,epsfig}
\usepackage{amsfonts,epsfig,a4wide}
\usepackage{amsfonts,a4wide}
\usepackage{graphics}

\usepackage{amsmath,amssymb,amsthm}
\usepackage{url}

\usepackage{euscript,color}

\newtheorem{thm}{Theorem}
\newtheorem{prop}{Proposition}

\newtheorem{lem}{Lemma}

\newtheorem{rem}{Remark}
\newtheorem{Prob}{Problem}

\newtheorem{Ex}{Example}
\newtheorem{Def}{Definition}

\newcommand{\R}{\mathbb{R}}

\newcommand{\omo}{\operatorname{\gamma}}

\newcommand{\weg}[1]{}

\usepackage{url,hyperref}

\begin{document}

\title{On the extendability of parallel sections of linear connections}

\author{Antonio J. Di Scala\thanks{Dipartimento di Scienze Matematiche ``G.L. Lagrange", Politecnico di Torino,
Corso Duca degli Abruzzi 24, 10129 Torino (Italy), {\tt antonio.discala@polito.it}}, Gianni Manno\thanks{Dipartimento di Matematica, Universit\`a degli
    studi di Padova, Via Trieste, 63, 35121 Padova, Italy,
    \texttt{gianni.manno@math.unipd.it}}}

\date{\today}
\maketitle

\begin{abstract}
Let $\pi:E\to M$ be a vector bundle over a simply connected manifold and $\nabla$ a linear connection in $\pi$. Let $\sigma: U \rightarrow E$ be a $\nabla$-parallel section of $\pi$ defined on a connected open subset $U$ of $M$. We give sufficient conditions on $U$ in order to extend $\sigma$ to the whole $M$. We mainly concentrate to the case when $M$ is a $2$-dimensional simply connected manifold.

\end{abstract}



\section{Introduction}
Many interesting problems in Differential Geometry can be formulated in terms of the existence of \emph{parallel sections}
of suitable defined linear connections in some vector bundles.
For example, if $(M,g)$ is a pseudo-Riemannian manifold, the existence of a Killing vector field turns out to be equivalent to the existence of a parallel section of a connection $\widetilde{\nabla}$ in the vector bundle $TM\oplus\mathfrak{so}(TM)$ over $M$ (which is canonically isomorphic to $TM\oplus\Lambda^2M$ through the metric) defined as follows (see also \cite{BFP02,CO09,Ko55}):
\begin{equation}\label{eq.conn.Kostant}
\widetilde{\nabla}_X(Y,\mathcal{A}):= \bigg( \nabla_XY-\mathcal{A}(X),\,\nabla_X\mathcal{A}-\mathrm{R}(X,Y)\bigg)
\end{equation}
where $X$ is a vector field on $M$, $\nabla_X$ is the covariant derivative associated with the Levi-Civita connection of $(M,g)$ and $\mathrm{R}$ the curvature.

\smallskip\noindent
More generally, also a projective vector field, i.e. a vector field which preserves geodesics as unparametrized curves, is a parallel section of a suitable linear connection in the adjoint \emph{tractor bundle} associated with the projective structure \cite{private} (the general theory
behind this is developed in \cite{C08}, where the projective case is briefly discussed). In the area of the tractor calculus we can find more interesting examples.
For instance, the property of $(M, g)$ of being conformally flat is equivalent to the flatness of the tractor connection \cite{BEG94}. More generally, the existence of an Einstein metric in the conformal class $[g]$ amounts to find a parallel section of such connection \cite{G10}.

\noindent
Sometimes it is possible to show the existence of a parallel section on an open dense subset $U$ of the base manifold $M$ and the question whether it is possible to extend it to the whole $M$ naturally arises.
Standard examples (see Section \ref{sec.counter}) show that this is not possible in general so that suitable assumptions must be made on $U$ and $M$. More often it is possible to prove the existence of a parallel section in a neighborhood of almost every point of $M$, so that the problem of gluing  together such sections in a global one naturally appears. Such situation comes out, for instance, when studying the singular (local) action of the Lie algebra of projective vector fields on a $2$-dimensional pseudo-Riemannian manifold $(M,g)$. Indeed, in \cite{BMM08} it has been proved that the set of points on which such action is locally regular (i.e. points possessing a neighborhood foliated by orbits of constant dimension) forms an open dense subset of $M$ and that a Killing vector field (that we recall is a parallel section of connection \eqref{eq.conn.Kostant}) exists in a neighborhood of any regular point.

\smallskip\noindent
Motivated by the above questions we propose the following problem.

\begin{Prob}\label{problem}
Let $\pi:E\to M$ be a vector bundle over a simply connected manifold $M$ and $\nabla$ a linear connection in $\pi$.
Let  $\sigma: U \rightarrow E$ be a non-zero $\nabla$-parallel section defined on a connected, open and dense subset $U$ of $M$. Does $\sigma$ extend to a parallel section on the whole $M$?
\end{Prob}

\noindent
There are two cases in which the above problem admits a not difficult positive solution.
One of them is the case when $\pi$ is a rank-one vector bundle. In such a case, the existence of a non-zero parallel section on a dense subset implies that the curvature tensor $\mathrm{R}^{\nabla}$ vanishes identically, so that the connection is flat and $\sigma$ can be extended on the whole $M$. The second case is when $\pi$ has rank-two and it is provided with a metric compatible with $\nabla$. In this case, the existence of a parallel section $\sigma$ on a dense subset implies that $\mathrm{R}^{\nabla}$ vanish identically and again $\sigma$ can be extended on the whole $M$.

\smallskip\noindent
In Section \ref{sec.counter} we give examples showing that the hypothesis of connectedness of $U$ and simply-connectedness of $M$ cannot be removed. Anyway, if the complement $M \setminus U$ has higher codimension we can drop the hypothesis of simply-connectedness. Namely, our first result is contained in the following theorem.

\begin{thm}\label{higherCodimension}Let $\pi:E\to M$ be a vector bundle over a connected manifold $M$ and $\nabla$ a linear connection in $\pi$.
Let  $\sigma: U \rightarrow E$ be a non-zero $\nabla$-parallel section defined on an open subset $U$ of $M$ whose complement $F:=M \setminus U$ is contained in a smooth submanifold of codimension greater or equal to $2$.
Then $\sigma$ can be extended on the whole $M$ as a parallel section.
\end{thm}

\noindent
As an application of this result we give a simple solution to a problem discussed by R. Bryant in MathOverflow \cite{BryantOver} concerning the existence of a Killing vector field defined on a compact Riemann surface $M$ without a finite number of points  (see Remark \ref{rem:SurSeg}).

\smallskip\noindent
In Section \ref{section:RR+} we introduce the conditions $\mathcal{R}$ and $\mathcal{R}^+$ for an open subset $U \subset \mathbb{R}^2$. We prove that under condition $\mathcal{R}$ (resp. $\mathcal{R}^+$) a $\nabla$-parallel section $\sigma:  U \rightarrow E$ of a vector bundle $E\to\mathbb{R}^2$, where $\nabla$ is a metric (resp. general) connection, can be extended on the whole $\mathbb{R}^2$. More precisely, here is our second result.

\begin{thm}\label{piano}
Let $\pi:E\to \mathbb{R}^2$ be a vector bundle endowed with a linear connection $\nabla$.
Let $\sigma :  U \rightarrow E$ be a $\nabla$-parallel section defined on the open and dense subset $U \subset \mathbb{R}^2$.
Then $\sigma$ can be extended on the whole $\mathbb{R}^2$ as a $\nabla$-parallel section if at least one of the following conditions holds:
\begin{enumerate}
\item[(i)]  there exists a $\nabla$-parallel metric $g$ on $E$ and the domain $U$ of $\sigma$ satisfies condition $\mathcal{R}$;
\item[(ii)]  the domain $U$ of $\sigma$ satisfies condition $\mathcal{R}^+$.
\end{enumerate}
\end{thm}

\smallskip\noindent
In section \ref{section:Compact} we show that the condition $\mathcal{R}$ holds when the complement $\mathbb{R}^2 \setminus U$ is compact and has zero Lebesgue measure.

\noindent
As an interesting application of the result contained in the item $(ii)$ of Theorem \ref{piano}, we can prove that a Killing vector field defined on a Riemann surface minus a segment can be always extended on the whole Riemann surface.
Here by a segment we mean a segment in some coordinate system.
Such extension can be also obtained by using a radial extension  (see Remark \ref{rem:SurSegI}).

\begin{rem}
By the uniformization theorem, a $2$-dimensional simply connected manifold $M^2$ is either the plane $\mathbb{R}^2$ or the
sphere $S^2$. To recover our extension theorem in the case of a vector bundle $E\to\mathbb{S}^2$ we can remove a point of $U$ and restrict
the vector bundle to $\mathbb{R}^2$.
\end{rem}

\noindent
In Section \ref{sec.killing.conformal} we see that, even though the obtained results can be used for extending projective (in particular Killing, affine, homothetic) vector fields, they cannot apply to the class of conformal vector fields. The point here is that such vector fields, in the $2$-dimensional case, are not parallel sections of a vector bundle endowed with a linear connection.

\smallskip\noindent
Finally, in Section \ref{section:Kostant} we study the extendibility of Killing vector fields by using the Kostant connection.

\smallskip\noindent
It is worth mentioning that the results of the present paper have been used in \cite{GPW} in the context of tractor connections.

\section{Why the hypotheses of connectedness of Problem \ref{problem} are essential?}\label{sec.counter}

The examples below show that the hypotheses of connectedness of $U$ and of simply-connectedness of $M$ in Problem \ref{problem} are essential. To start with, we show why the hypothesis on the connectedness of $U$ cannot be removed.

\begin{Ex} \label{sconesso}
Let $\mathbb{R}^2$ be the standard Euclidean space with coordinates $(x,y)$. Let $X$ be the Killing vector field (that we recall is a particular parallel section of the connection $\widetilde{\nabla}$ defined by \eqref{eq.conn.Kostant})  defined on $U=\mathbb{R}^2\setminus\{(x,y)\,|\,y=0\}$ as follows: in the half-plane $y>0$, $X$ is the right translation and in the half-plane $y<0$ is the left translation (see Figure \ref{fig.translation}). It is obvious that $X$ cannot be extended on the whole $\mathbb{R}^2$.
\begin{figure}[h!]
\begin{center}
\includegraphics[width=0.7\textwidth]{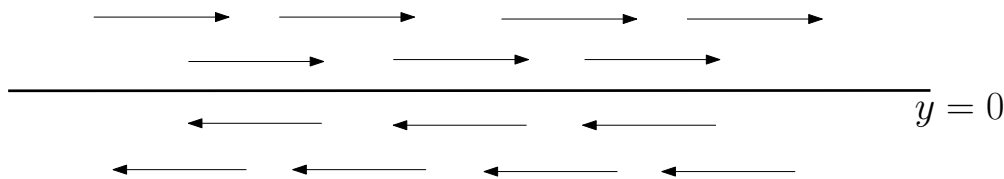}
\end{center}
\caption{In general, a Killing vector field defined on a disconnected open dense subset is not extendable to its closure.}\label{fig.translation}
\end{figure}
\end{Ex}

\noindent
Now we list several examples regarding the hypothesis on $M$ of being simply connected.

\begin{Ex} \label{S1}
Let $E := \mathbb{S}^1 \times \R$ be the trivial vector bundle over the circle $\mathbb{S}^1$ and $\mathbf{e}:p\in\mathbb{S}^1 \to (p,1)$ a section of $E$.
The non-exact $1$-form $\mathrm{d} \theta$ on $\mathbb{S}^1$, where $\theta$ is the angle coordinate, gives rise to a connection $\nabla$ in $E$, i.e., the derivative of the section $\mathbf{e}$ is given by
\[
\nabla \mathbf{e} := \mathrm{d} \theta \otimes \mathbf{e} \, .
\]
Let $p\in\mathbb{S}^1$. The complement $U:=\mathbb{S}^1\setminus \{p\}$ is an interval and by parallel transport there exists a non-zero $\nabla$-parallel section $\sigma: U \rightarrow E$. Since the equation $\frac{df}{d\theta} + f = 0$ has no non-zero periodic solutions, $\sigma$ cannot be extended on $\mathbb{S}^1$.
\end{Ex}

\begin{Ex} \label{nastro}
Let $\Sigma \subset \mathbb{R}^3$ be a Moebius strip in the Euclidean space, i.e., the standard example of non-orientable surface. Let $\nu(\Sigma)$ be its normal bundle endowed with the normal connection $\nabla^{\perp}$. As it is well-known, by removing the central circle $\gamma$ of $\Sigma$ we get a cylinder, which is connected and orientable open subset $U:= \Sigma \setminus \gamma$. So, the restriction to $U$ of normal bundle $\nu(\Sigma)$ admits a $\nabla^{\perp}$-parallel section $\sigma$. Since the Moebius strip is not orientable, the section $\sigma$ cannot be extended to the whole $\Sigma$.
\end{Ex}

\noindent
Another way of constructing examples is that to use a flat connection in a vector bundle over a compact Riemann surfaces $\Sigma$ of genus $\geq 1$.
Indeed, as it is well-known (see e.g. \cite[page 559]{AB83}), an irreducible representation $\rho$ of $\pi_1(\Sigma)$ gives rise to a flat connection $\gamma$ whose holonomy group is $\rho(\pi_1(\Sigma))$. Thus, such connection has no global parallel sections. On the other hand a compact Riemann surface
is obtained from a polygon with edges glued together in pairs. The interior $U$ of such polygon is simply connected so that the restriction of $\gamma$ to $U$ gives a flat bundle over $U$, implying the existence of a globally defined $\nabla$-parallel section $\sigma$ on $U$.
As explained above, the section $\sigma$ cannot be extended to the whole $\Sigma$.

\section{The conditions $\mathcal{R}$ and $\mathcal{R}^+$}\label{section:RR+}

Here is the definition of condition $\mathcal{R}$.

\begin{Def}\label{conditionR}
An open subset $U \subset \mathbb{R}^2$ satisfies condition $\mathcal{R}$ if there exists
a point $p_0 \in U$ and a dense subset $V \subset U$ such that for any $p \in V$ there is a compact subset $K_p$ containing the segment $\overline{p_0p}$
such that for any $\epsilon > 0$  there exist disjoint subsegments $I_i\subset\overline{p_0p}$ and
piecewise smooth curves $\gamma_i \subset U$, $i=1,\dots,n$ with the following properties:
$\overline{p_0p}\setminus \bigcup_i I_i\subset U$ and the concatenation of $I_i$ and $\gamma_i$ forms, for each $i\in\{1,\dots,n\}$, a continuous piecewise smooth Jordan curve bounding a region $S_i$ contained in the compact subset $K_p$ such that
\[
\sum_{i=1}^{n} \mu(S_i) \leq \epsilon
\]
where $\mu$ is the Lebesgue measure.
\end{Def}

\begin{figure}[h!]
\begin{center}
\includegraphics[scale=0.5]{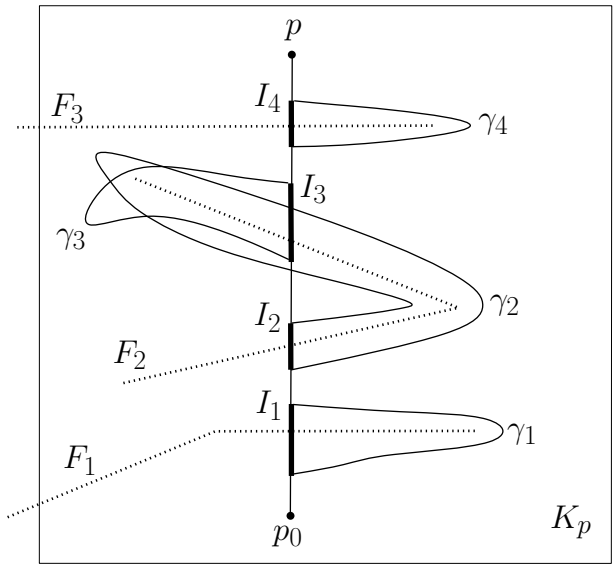}
\end{center}
\caption{Example of condition $\mathcal{R}$ for the complement $U$ of the closed set $F = F_1 \cup F_2 \cdots \cup F_4$. }\label{fig.conditionR}
\end{figure}

\newpage

\noindent
Here is the definition of condition $\mathcal{R}^+$.
\begin{Def}\label{conditionR+}
%
An open subset $U \subset \mathbb{R}^2$ satisfies condition $\mathcal{R}^+$ if there exists
a point $p_0 \in U$ and a dense subset $V \subset U$ such that for any $p \in V$ there is a compact subset $K_p$ containing the segment $\overline{p_0p}$
such that for any $\epsilon,G > 0$  there exist disjoint subsegments $J_i\subset\overline{p_0p}$ and
piecewise smooth curves $\gamma_i \subset U$, $i=1,\dots,n$ with the following properties:
$\overline{p_0p}\setminus \bigcup_i J_i\subset U$ and the concatenation of $J_i$ and $\gamma_i$ forms, for each $i\in\{1,\dots,n\}$, a continuous piecewise smooth Jordan curve bounding a region $S_i$ contained in the compact subset $K_p$ such that
\[
e^{\frac{\mathrm{G}}{2}\mathrm{L}_{\gamma}}\sum_{i=0}^{n-1} e^{\mathrm{G} \mathrm{L}_i } \mu(S_i) < \epsilon
\]
where $\mu$ is the Lebesgue measure, $\mathrm{L}_{\gamma} $ is the sum of the length of the curves $\gamma_i$ and
$\mathrm{L}_i$ is the maximal length of the curves of a homotopy of the region $S_i$ relative to the endpoints of $J_i$ deforming $\gamma_i$ to the interval $J_i$.
\end{Def}

\subsection{Example: the property $\mathcal{R}^+$ is satisfied by the complement of a segment}
\label{sotto:segmento}

Here we show that the open subset $U := \mathbb{R}^2 \setminus I$, where $I$ is a segment, satisfies the property $\mathcal{R}^+$.
Without loss of generality we can assume that $I$ is the interval $[0,1]$ in the $x$-axis i.e. $I = \{ (x,0) :  0 \leq x \leq 1 \}$.
Let $p_0 = (\frac{1}{2},\frac{1}{2}) \in U$.
Set $V=U$ and let $p=(p_x,p_y) \in V=U$. Notice that if the segment $\overline{p_0p}$ is disjoint from $I$ then property $\mathcal{R}^+$ holds trivially.
So assume that $\overline{p_0p} \cap I$ is not empty.
Let $K_p \subset \mathbb{R}^2$ be a closed ball centered
at $(0,0)$ whose interior contains both the interval $I$ and the point $p$.
Observe that $p_0$ is an interior point of $K_p$.
Let $G$ and $\epsilon$ be as in the definition of property $\mathcal{R}^+$.
Let $T_{\delta}$ be the set of points at distance $\leq \delta$ from $I$. Namely,
\[
T_{\delta} = \{ q \in \mathbb{R}^2 : \mathrm{dist}(q,I) \leq \delta \}
\]
Fix $\delta<\min\{\frac{1}{2},p_y\}$ small enough such that $T_{\delta}$ is contained in the interior of $K_p$.
Let $p_1,p_2$ be the two points of the intersection of the segment $\overline{p_0p}$ with the boundary $\partial T_{\delta}$ of the set $T_{\delta}$ (being $p_1$ closest to $p_0$). Set $p_3 = p$ and $\gamma_0 = \overline{p_0p_1}$. The curve $\gamma_1$ is one of the two connected components of $\partial T_{\delta}\setminus\{p_1,p_2\}$. Define $\gamma_2$ as $\gamma_2 := \overline{p_2p_3}$.
Now $\mu(S_0) = \mu(S_2) = 0$ and \[ \mu(S_1) \leq \mu(T_{\delta}) = 2 \delta + \pi \delta^2 \, ;\]
\[
L_{\gamma} \leq \mathrm{dist}(p,p_0) + \mathrm{perimeter}(T_{\delta})  = \mathrm{dist}(p,p_0) + 2 + 2\pi \delta.
\]
Set $h(x,s) := x (s p_2 + (1-s)p_1 ) + (1-x) \gamma_1(s) $ where $\gamma_1(s)$ is a parametrization of $\gamma_1$ from $p_1$ to $p_2$ where $(x,s) \in [0,1] \times [0,1]$. Then $h$ is a homotopy as in property $\mathcal{R}^{+}$ deforming $\gamma_1$ into the interval $\overline{p_1p_2}$. Moreover since $S_1$ is convex we see that the curves $h(x, \cdot)$ of the homotopy are always contained in $S_1$. The length of the curves $h(x, \cdot)$ are always bounded by the length of $\gamma_1$. Then
\[
e^{\frac{G}{2} L_{\gamma}}  e^{G L_1} \mu(S_1) \leq e^{\frac{G}{2} (\mathrm{dist}(p,p_0) + 2 + 2 \pi \delta)}
 e^{G (2 + 2 \pi \delta)} (2 \delta + \pi \delta^2) \, .
 \]
Now is clear that for $\epsilon > 0$ there exists a $\delta > 0$ such that
\[
e^{\frac{G}{2} L_{\gamma}}  e^{G L_1} \mu(S_1) \leq e^{\frac{G}{2} (\mathrm{dist}(p,p_0) + 2 + 2 \pi \delta)}
 e^{G (2 + 2 \pi \delta)} (2 \delta + \pi \delta^2) < \epsilon \, .
 \]
This shows that the complement of a segment $I$ has the property $\mathcal{R}^{+}$.

\subsection{The case when $\mathbb{R}^2 \setminus U $ is compact and has zero Lebesgue measure}\label{section:Compact}

Here we prove the following proposition.

\begin{prop} Let $U \subset \mathbb{R}^2$ be an open and connected subset whose complement $F=\mathbb{R}^2 \setminus U$ has zero measure and is compact. Then $U$ satisfies property $\mathcal{R}$.
\end{prop}
\begin{proof}
Let us fix a  point $p_0\in U$.  Let $p \in U$ be any other point. Since $F$ is compact there exists a disc $K_p$ containing $F$ and the segment $\overline{p_0p}$ in its interior. Let $\epsilon$ be small enough such that $F$ can be covered with an union $\mathbf{D}$ of discs whose measure $\mu(\mathbf{D})$ is smaller than $\epsilon$ and contained in $K_p$. Since
$F$ is compact, we can assume that $\mathbf{D}$ is a finite union of discs. The set $\mathbf{D}$ has a finite number of connected components.
Assume now that the connected components of $\mathbf{D}$ intersecting the segment $\overline{p_0p}$ are simply connected.
By starting at $p_0$ and moving along the segment $\overline{p_0p}$ we will meet a first point $a_1$ belonging to the boundary of one of the components of $\mathbf{D}$. In the case when $a_1$ belongs to several connected components we just select one of them and call it $\mathbf{D}_1$.
Since $\mathbf{D}_1$ is simply connected, by following its boundary, we will meet the segment $\overline{p_0p}$ at a point $b_1 \in \overline{p_0p}$ such that both
the interval $I_1 = \overline{a_1 b_1}$ and the boundary curve $\gamma_1$ from $a_1$ to $b_1$ are as in condition $\mathcal{R}$, i.e. their concatenation $I_1 \sharp \gamma_1$ is a continuous piecewise smooth Jordan curve bounding a region $S_1$ contained in $\mathbf{D}_1$. Now, by starting at the point $b_1$ and moving towards $p$ along $\overline{p_0p}$, we will meet another point $a_2$ of the boundary of one of the connected components of $\mathbf{D}$. As we did for $a_1$ we can do for $a_2$. Namely, by following the boundary of the respective connected component, we will get another point $b_2 \in \overline{p_0p}$ such that both the interval $I_2 = \overline{a_2 b_2}$ and the boundary curve $\gamma_2$ from $a_2$ to $b_2$ are as in condition $\mathcal{R}$, i.e. their concatenation $I_2 \sharp \gamma_2$ is a continuous piecewise smooth Jordan curve bounding a region $S_2$ contained in the respective connected component.
Then, by starting at $b_2$, we can repeat the above argument to construct a finite sequence of intervals $I_i$ and boundary curves $\gamma_i$, $i=1,\cdots,N$ as in condition $\mathcal{R}$. Since all the regions $S_i$ are included in $\mathbf{D}$, we get that \[  \sum_{i=1}^{N} \mu(S_i) \leq \mu(\mathbf{D}) \leq \epsilon \]
showing that, under the hypothesis that all the connected components of  $\mathbf{D}$ intersecting the segment $\overline{p_0p}$ are simply connected, condition $\mathcal{R}$ holds.

\smallskip\noindent
Now assume that a connected component $\mathbf{A}$ of $\mathbf{D}$ is not simply connected. The homotopy type of $\mathbf{A}$ is that of a bouquet of a finite number of circles as $\mathbf{A}$ is a finite union of discs. Since the set $U$ is connected, we can cut each one of the circles of the bouquet $\mathbf{A}$ so that to obtain a new set $\mathbf{D}$ which covers $F$ and having all connected components also simply connected.
Then we can apply the previous argument. This completes the proof.
\end{proof}

\begin{rem}
Notice that the hypothesis of compactness on $F$ is just used to find a simply connected compact set $K_p$ containing all the connected components of the intersection of $F$ and the segment $\overline{p_0p}$ in its interior. So the above proof applies also to non-compact subsets $F$ whose all connected components are compact.
\end{rem}

\section{Proof of Theorem \ref{higherCodimension}}

Here is the proof of Theorem \ref{higherCodimension}.
\begin{proof}
Since the domain of $\sigma$ is assumed to be dense it is enough to show that $\sigma$ can be extended around any point of $F$.
So let $p \in F$ be a point where $\sigma$ is not defined. Let $S \subset M$ be the submanifold of codimension $\geq 2$ which contains $F$. Then near $p$ we can find a coordinate system $(x_1,\cdots,x_m)$  centered at $p$ of $M$ such that $S$ is locally described by system $\{x_1=x_2=\cdots=x_{m-s}=0\}$, where $s=\mathrm{dim}(S)$.
Let $q \in M$ be a point whose coordinates $x_1,\cdots,x_m$ are $(\epsilon, 0, 0, \cdots, 0)$ with $\epsilon$ small enough such that an open ball $B_q$ center at $q$ is contained in the coordinate system $x_1,\cdots,x_m$ and $p \in B_q$.
Observe that $q$ belongs to the open subset $U$. Consider the smooth section $\widetilde{\sigma}$ defined on $B_q$ by parallel transporting $\widetilde{\sigma}(q) := \sigma(q)$
along the radial lines through the point $q$. If $\mathcal{L}$ is a radial line through the point $q$ which does not intersect $S$, then we have
\[
\widetilde{\sigma}|_{\mathcal{L}} = \sigma|_{\mathcal{L}} \, .
\]
Observe that the subset $G \subset B_q$ of points $x \in B_q$ such that the radial line $\overline{xq}$ does not intersect $S$ is dense in $B_q$. Indeed, the radial lines through $q$ which intersect $S$ are contained in the intersection of $B_q$ with the hyperplane $x_2 = 0$.  Then we have \[ \widetilde{\sigma}|_{B_q \bigcap U} = \sigma|_{B_q \bigcap U} \, \]
Since $B_q \bigcap U$ is dense in $B_q$ we get that $\widetilde{\sigma}$ is a $\nabla$-parallel section on $B_q$.
Now it is clear that $U' := U \bigcup B_q$ extends the domain of definition of $\sigma$ as a $\nabla$-parallel section.
\end{proof}


\begin{rem}\label{rem:SurSeg} The above result can be used to give a different solution of a problem discussed by  R. Bryant about the extension of a Killing vector field defined on a Riemannian surface minus a finite number of points
\url{http://mathoverflow.net/questions/122438/compact-surface-with-genus-geq-2-with-killing-field}

\end{rem}

\section{Proof of Theorem \ref{piano}}

The idea of the proof of Theorem \ref{piano} is to use a suitable estimate, involving the curvature of the connection, to control the parallel transport along curves. To this aim, we need the following lemma.
\begin{lem}\label{lemma.estimate}
Let $a\in\mathbb{R}$. Let $f(t)$ and $g(t)$ be two continuous functions for $t\geq a$. Let $u(t)$ be a $C^1$ function for $t\geq a$. If
\begin{equation}\label{eq.ineq}
\left\{
\begin{array}{l}
u'(t)\leq f(t)u(t)+g(t)\,, \quad t\geq a
\\
u(a)=u_0
\end{array}
\right.
\end{equation}
then
\begin{equation}\label{eq.ineq.2}
u(t)\leq u_0e^{\int_a^tf(x)dx} + \int_a^tg(s)e^{\int_s^tf(x)dx}ds.
\end{equation}
We underline that the right-hand side term of \eqref{eq.ineq.2} is the solution to the Cauchy problem given by \eqref{eq.ineq} if we have the equality in that system.
\end{lem}
\begin{proof}
A direct computation shows that \eqref{eq.ineq} can be written as
$$
\frac{d}{ds}\left(u(s)e^{\int_s^tf(x)dx}\right)\leq g(s)e^{\int_s^tf(x)dx}\,, \quad s\geq a\,, \quad u(a)=u_0
$$
and integrating over $s$ from $a$ to $t$ we obtain \eqref{eq.ineq.2}.
\end{proof}

\begin{prop}\label{prop.estimate}
Let $\pi:E\to \mathbb{R}^2$ be a vector bundle endowed with a linear connection $\nabla$ and $g$ be a metric on $\pi$ (not necessarily compatible with the connection $\nabla$).  Let $\omo_0$ and $\omo_1$ be two curves starting at $p \in \mathbb{R}^2$ and ending at $q \in \mathbb{R}^2$. Let $\omo : [0,1] \times [0,1] \rightarrow \mathbb{R}^2$ be a smooth homotopy between $\omo_0$ and $\omo_1$ relative to the endpoints $p,q$ which is $1$-$1$ when restricted to $(0,1) \times [0,1]$. Let $S := \omo([0,1] \times [0,1])$. Then
\begin{equation}
\|\tau_{\gamma_0}(\xi_p)-\tau_{\gamma_1}(\xi_p)\|_g \leq \|\xi_p\|_g \, \mathrm{R} \, e^{\mathrm{G}\mathrm{L}} \mu(S)
\end{equation}
where $\tau_{\gamma_i}(\xi_p)$ is the parallel transport from $p$ to $q$ of $ \xi_p \in\pi^{-1}(p)$ along $\gamma_i$, $\mu(S)$ is the area of $S$ w.r.t. the Lebesgue measure of $\mathbb{R}^2$, $\mathrm{R}$ is a constant depending only on the metric $g$ and on the curvature tensor $\mathrm{R}^{\nabla}$ of $\nabla$ on $S$,  $\mathrm{L}= \underset{s \in [0,1]}{\mathrm{max}} \{ \mathrm{length}(\gamma_s) \}$, $\gamma_s(t):=\gamma(t,s)$, is the maximal length of the curves of the homotopy $\omo$ and $\mathrm{G}$ is a constant controlling the norm of the tensor $\nabla g$ on $S$, i.e. the constants $\mathrm{G}$ and $\mathrm{R}$ depends only on the image $S$ of the homotopy and not on the homotopy itself.
\end{prop}
\begin{proof}
We denote by $\|\xi\|_g := g(\xi,\xi)$ the norm of the vector $\xi_p$ of the fiber $E_p = \pi^{-1}(p)$.
If $v$ is a tangent vector of $\mathbb{R}^2$, its norm $\|v\|$ is taken w.r.t. the flat standard Riemannian metric, i.e. $\|v\|$ is the length of the vector $v$.

\noindent
Regarding the curvature tensor $\mathrm{R}^{\nabla}$ as a map $\mathrm{R}^{\nabla} : \Lambda^2{T_{(x,y)} \mathbb{R}^2} \rightarrow \mathrm{End} (E_{(x,y)})$ we have
\[
g(\mathrm{R}^{\nabla}(v \wedge w) \eta , \xi) = g(\mathrm{R}^{\nabla}(v,w) \eta, \xi)
\]
with $v,w \in T_{(x,y)}\mathbb{R}^2$ and $\eta, \xi \in E_{(x,y)}$. Since $S$ is compact, there exists a constant $\mathrm{R}$ such that
\[ g(\mathrm{R}^{\nabla}(v \wedge w) \eta , \xi) \leq \mathrm{R} \| v \wedge w \| \,  \|\xi\|_g \,  \|\eta\|_g \]
for all tangent vectors $v,w$ of $S$ and $\eta, \xi \in \pi^{-1}(S)$, where $\| v \wedge w \|$ is the area of the parallelogram  spanned by $v,w$.

\noindent
We denote by $\partial_t$ and $\partial_s$, respectively, the vector fields $\frac{\partial \omo}{\partial t}$ and  $\frac{\partial \omo}{\partial s}$, both tangent to $S$ at the point $\gamma(t,s)$. Let us define  $X(t,s)$ as the parallel transport of $\xi_p \in\pi^{-1}(p)$ along $\gamma_s$ at the instant $t$ (see Figure \ref{fig.transport}). We have
\begin{equation}\label{eq.first.ineq}
\|\tau_{\gamma_0}(\xi_p)-\tau_{\gamma_1}(\xi_p)\| _g = \| X(1,1)-X(1,0)\| _g \leq \int_0^1\left\| \frac{D}{ds}X(1,s) \right\| _g ds.
\end{equation}
The symbol $\frac{D}{ds}X(t,s)$ stands for the covariant derivative along the curve $s\to\gamma_t(s) := \gamma(t,s)$ (i.e. $t$ is fixed) associated with $\nabla$. Thus, for $(t,s) \in (0,1) \times (0,1)$
$\frac{D}{ds}X(t,s) =  \nabla_{\partial_s} X(t,s) $ and $\frac{D}{ds}X(1,s) = \frac{\partial X(1,s)}{\partial s}$ is the derivative in the vector space $E_q$ of the curve $X(1,s) \in E_q$, see Chapter 2 of \cite{DoC92} for details.
So, the above estimate is obtained by applying the fundamental theorem of the integral calculus.
\begin{figure}[h!]
\begin{center}
\includegraphics[width=0.5\textwidth]{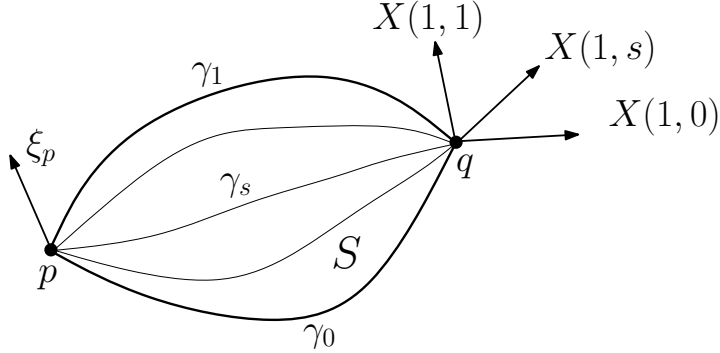}
\end{center}
\caption{$X(t,s)$ is constructed by parallel transporting $v$ along $\gamma_s$.}
\label{fig.transport}
\end{figure}

\noindent
The tensor $(\nabla_v g)(\xi,\eta) := v \big(g(\xi,\eta)\big)- g(\nabla_v \xi, \eta) - g(\xi, \nabla_v \eta)$ is continuous so that by the compactness of $S$ there exists a constant $\mathrm{G}$ such that
\[
(\nabla_v g)(\xi,\eta) \leq \mathrm{G} \|v\| \, \|\xi\|_g \, \| \eta \|_g
\]
where $\|v\|$ is the norm of the tangent vector $v$ of $S$ and $\eta, \xi \in \pi^{-1}(S)$. Then
$$
{\partial_t}\|X\|_g^2=\nabla_{\partial_t}g(X,X)\leq \mathrm{G} \|\partial_t\| \, \|X\|_g^2
$$
so that, in view of Lemma \ref{lemma.estimate}, we obtain
\begin{equation}\label{norma}
\|X(t,s)\|_g \leq \|\xi_p\|_g \left(e^{\int^t_0 \mathrm{G} \|\partial_t\|_{(t',s)} dt')}\right)^{\frac{1}{2}}\leq \|\xi_p\|_g e^{\frac{\mathrm{G} \,  \mathrm{L}}{2}}
\end{equation}
where $\mathrm{L}= \underset{s \in [0,1]}{\mathrm{max}} \{ \mathrm{length}(\gamma_s) \}$.

\noindent
Now we want to estimate $\left\|\frac{DX}{ds}\right\|_g$.  From the equation
\[
{\partial_t}\left\|\frac{DX}{ds}\right\|^2=2g\left( \frac{D}{dt}\frac{DX}{ds}\,,\,\frac{DX}{ds}\right)+ \nabla_{\partial_t}g\left(\frac{DX}{ds}\,,\,\frac{DX}{ds}\right)
=
2g\left( \mathrm{R}^\nabla(\partial_t,\partial_s)X\,,\,\frac{DX}{ds}\right)+ \nabla_{\partial_t}g\left(\frac{DX}{ds}\,,\,\frac{DX}{ds}\right)
\]
and the above inequalities we get
\[
{\partial_t}\left\|\frac{DX}{ds}\right\|_g^2 \leq 2\mathrm{R}\,  \| \partial_t \wedge \partial_s\| \, \, \|\xi_p\|_g e^{\frac{\mathrm{G} \,  \mathrm{L}}{2}} \, \left\|\frac{DX}{ds}\right\|_g + \mathrm{G} \, \|\partial_t\| \, \left\|\frac{DX}{ds}\right\|_g^2
\]
which implies
\[ {\partial_t} \left\|\frac{DX}{ds}\right\|_g \leq \mathrm{R}\,
\| \partial_t \wedge \partial_s\| \, \, \|\xi_p\|_g e^{\frac{\mathrm{G} \,  \mathrm{L}}{2}}  + \frac{\mathrm{G \, \|\partial_t\|}}{2} \left\|\frac{DX}{ds}\right\|_g \, .
\]
By Lemma \ref{lemma.estimate} we obtain
\[
\left\|\frac{DX(t,s)}{ds}\right\|_g \leq  \int_0^t \mathrm{R}\,  \| \partial_t \wedge \partial_s\|_{(t',s)} \, \, \|\xi_p\|_g e^{\frac{\mathrm{G} \,  \mathrm{L}}{2}} e^{\left(\int_{t'}^{t} \frac{\mathrm{G} \, \|\partial_t\|_{(t'',s)}}{2} dt''\right)}  d t'
\]
and so
\[
\left\|\frac{DX(t,s)}{ds}\right\|_g \leq  \|\xi_p\|_g \, \mathrm{R} \, e^{\mathrm{G}\mathrm{L}} \int_0^t \,  \| \partial_t \wedge \partial_s\|_{(t',s)} \, \, dt'  \, .
\]
Finally from equation (\ref{eq.first.ineq}) we have
\[ \|\tau_{\gamma_0}(\xi_p)-\tau_{\gamma_1}(\xi_p)\|_g\leq \int_0^1\left\| \frac{D}{ds}X(1,s) \right\|_g ds \leq
\|\xi_p\|_g \, \mathrm{R} \, e^{\mathrm{G}\mathrm{L}}\int_0^1 \int_0^1 \,  || \partial_t \wedge \partial_s|| \, \, dt ds \]
which is we wanted to show.
\end{proof}

\begin{proof}[\bf Proof of Theorem \ref{piano}]

\noindent
Assume that condition $(i)$ of Theorem \ref{piano} holds. Let $p_0 \in U$ be the point given by condition $\mathcal{R}$ and
$\xi$ the smooth section defined on the whole $\mathbb{R}^2$ obtained by parallel transporting $\xi(p_0):= \sigma(p_0)$ along the radial straight lines starting at $p_0$. Note that $\sigma \equiv \xi$ near $p_0$.
We claim that $\sigma(p) = \xi(p) $ for all $p \in V$ where $V$ is a dense subset of the domain $U$ of $\sigma$.
Fix $p \in V$ and the compact $K_p$ containing the segment $\overline{p_0p}$ as in Definition \ref{conditionR}.
We relabel the segments $I_i$ (and the corresponding curves $\gamma_i$ and regions $S_i$) of Definition \ref{conditionR} in order to obtain a sequence of subsegments on the oriented segment $\overrightarrow{p_0p}$. Let now $a_i$ and $b_i$ be the endpoints of $I_i$.
The strategy is to apply the estimate of Proposition \ref{prop.estimate} to each region $S_i$.
Since $g$ is compatible with $\nabla$, the constant $\mathrm{G}$ which appears in Proposition \ref{prop.estimate} is zero.
Since the regions $S_i$ are inside the compact set $K_p$ we have an uniform bound $\mathrm{R}$ for the norm of the curvature tensor $\mathrm{R}^{\nabla}$ on $K_p$. We have that

\[ \begin{aligned} \| \xi(p) - \sigma(p) \|_g  &=  \| \xi(b_n) - \sigma(b_n) \|_g =  \| \tau_{I_{n}}\xi(a_n) - \tau_{\gamma_{n}}\sigma(a_n) \|_g  = \\
  &= \| \tau_{I_n}\xi(a_n) - \tau_{I_n} \sigma(a_n) +  \tau_{I_n} \sigma(a_n) -\tau_{\gamma_n} \sigma(a_n) \|_g \\
  &\leq \| \tau_{I_n}\xi(a_n) - \tau_{I_n} \sigma(a_n)  \|_g + \|  \tau_{I_n} \sigma(a_n) -\tau_{\gamma_n} \sigma(a_n) \|_g \\
  &\leq \| \xi(a_n) - \sigma(a_n)  \|_g +  \|  \tau_{I_n} \sigma(a_n) -\tau_{\gamma_n} \sigma(a_n) \|_g \\
    &\leq \| \xi(b_{n-1}) - \sigma(b_{n-1})  \|_g +  \|  \tau_{I_n} \sigma(a_n) -\tau_{\gamma_n} \sigma(a_n) \|_g \\
 \end{aligned} \]
Since the region $S_n$, whose boundary are the segment $I_{n}$ and the curve $\gamma_n$, is simply connected (by Definition \ref{conditionR}), we can use the Riemann mapping theorem to map $S_n$ in a $1$-$1$ way onto the unit disc $\Delta \subset \mathbb{R}^2$. Under such a mapping the segment $I_{n}$ and the curve $\gamma_{n}$ are mapped, respectively, into two complementary arcs $\delta_1$ and $\delta_2$ of the unit circle. Hence we can construct a smooth $1$-$1$ homotopy $\omo : [0,1]\times[0,1] \rightarrow \Delta $ for $\delta_1$ and $\delta_2$ relative to the endpoints of the arcs for $\delta_1$ and $\delta_2$. The pullback, by the Riemann mapping, of such homotopy is a homotopy between $I_{n}$ and $\gamma_{n}$. Then, by applying Proposition \ref{prop.estimate} to the region $S_{n}$, we obtain
\[ \| \xi(b_n) - \sigma(b_n) \|_g
  \leq \| \xi(b_{n-1}) - \sigma(b_{n-1}) \|_g +  \mathrm{R}\| \sigma(a_n)\|_g \mu(S_{n}) = \| \xi(b_{n-1}) - \sigma(b_{n-1}) \|_g +   \mathrm{R }\| \sigma(p_0)\|_g \mu(S_{n})
\]
By repeating the above argument for $j=n-1, \cdots, 1$ we  get
\[
\| \xi(b_j) - \sigma(b_j) \|_g
  \leq  \| \xi(b_{j-1}) - \sigma(b_{j-1}) \|_g +   \mathrm{R }\| \sigma(p_0)\|_g \mu(S_{j})
\]
and so
\[
\| \xi(p) - \sigma(p) \|_g  = \| \xi(p_n) - \sigma(p_n) \|_g \leq \mathrm{R }\| \sigma(p_0)\|_g \sum_{i=1}^{n}\mu(S_i) \leq \mathrm{R }\| \sigma(p_0)\|_g \, \epsilon \,.
\]
In view of the arbitrariness of $\epsilon$, $\xi= \sigma$ on $V \subset U$. Thus, $\xi \equiv \sigma$ on $U$ as $V$ is dense in $U$. Finally, $\xi$ is parallel on the whole $\mathbb{R}^2$ due to the fact that $U$ is a dense subset of $\mathbb{R}^2$. This proves the theorem under the hypothesis of item $(i)$.

\smallskip
\noindent
Now, assume that condition $(ii)$ of Theorem \ref{piano} holds. Let $p_0 \in U$ be the point given by condition $\mathcal{R}^+$. Let $g$ be any metric on the vector bundle $\pi:E\to\mathbb{R}^2$ such that $\| \sigma(p_0)) \|_g = 1$.
As in the previous case, let $\xi$ be the smooth section defined on the whole $\mathbb{R}^2$ obtained by parallel transporting $\xi(p_0):= \sigma(p_0)$ along the radial straight lines starting at $p_0$.
We shall prove that $\sigma(p) = \xi(p) $ for all $p \in V$ where $V$ is a dense subset of the domain $U$ of $\sigma$ of Definition \ref{conditionR+}.
Fix $p \in V$ and $K_p$ containing the segment $\overline{p_0p}$ as in Definition \ref{conditionR+}.
Let $\mathrm{G}$ be a bound for the norm of tensor $\nabla g$ on the compact subset $K_p$.
By using the same notations we introduced in the previous case, we have:
\[ \begin{aligned} \| \xi(p) - \sigma(p) \|_g  &=  \| \xi(p_n) - \sigma(p_n) \|_g =  \| \tau_{I_{n-1}}\xi(p_{n-1}) - \tau_{\gamma_{n-1}}\sigma(p_{n-1}) \|_g  = \\
  &= \| \tau_{I_{n-1}}\xi(p_{n-1}) - \tau_{I_{n-1}} \sigma(p_{n-1}) +  \tau_{I_{n-1}} \sigma(p_{n-1}) -\tau_{\gamma_{n-1}}\sigma(p_{n-1}) \|_g \\
  &\leq \| \tau_{I_{n-1}}\xi(p_{n-1}) - \tau_{I_{n-1}} \sigma(p_{n-1}) \|_g + \|  \tau_{I_{n-1}} \sigma(p_{n-1}) -\tau_{\gamma_{n-1}}\sigma(p_{n-1}) \|_g \\
  &\leq \| \xi(p_{n-1}) - \sigma(p_{n-1}) \|_g e^{\frac{\mathrm{G}}{2}\,\| p_n - p_{n-1}\|}+ \|  \tau_{I_{n-1}} \sigma(p_{n-1}) -\tau_{\gamma_{n-1}}\sigma(p_{n-1}) \|_g \\
  &\leq \| \xi(p_{n-1}) - \sigma(p_{n-1}) \|_g e^{\frac{\mathrm{G}}{2}\,\| p_n - p_{n-1}\|}+ \mathrm{R} \| \sigma(p_{n-1})\|_g e^{\mathrm{G} \mathrm{L}_{n-1}} \mu(S_{n-1}) \\
  &\leq \| \xi(p_{n-1}) - \sigma(p_{n-1}) \|_g e^{\frac{\mathrm{G}}{2}\,\| p_n - p_{n-1}\|}+ \mathrm{R} e^{\frac{\mathrm{G} \mathrm{L}_{\gamma}}{2}} e^{\mathrm{\mathrm{G}} \mathrm{L}_{n-1}} \mu(S_{n-1})
 \end{aligned} \]
  where the last two inequalities are obtained in view of Proposition \ref{prop.estimate} and inequality (\ref{norma}).\\
  By repeating the above argument for $j=n-1, \cdots, 1$ we  get
\[
\| \xi(p_j) - \sigma(p_j) \|_g \leq \| \xi(p_{j-1}) - \sigma(p_{j-1}) \|_g e^{\frac{\mathrm{G}}{2}\,\| p_j - p_{j-1}\|}+ \mathrm{R} e^{\frac{\mathrm{G} \mathrm{L}_{\gamma}}{2}} e^{\mathrm{\mathrm{G}} \mathrm{L}_{j-1}} \mu(S_{j-1})
 \, \, .
\]
Then
\[ \| \xi(p) - \sigma(p) \|_g  =  \| \xi(p_n) - \sigma(p_n) \|_g  \leq \mathrm{R}e^{\frac{\mathrm{G}}{2}\mathrm{L}_{\gamma}}e^{\frac{\mathrm{G}}{2}\|p-p_0\|}\sum_{i=0}^{n-1} e^{\mathrm{G} \mathrm{L}_i } \mu(S_i) \leq \mathrm{R} e^{\frac{\mathrm{G}}{2}\|p-p_0\|}\,\epsilon
\]
Since $\epsilon$ is arbitrary, $\xi(p) = \sigma(p)$ for any $p\in V \subset U$, hence $\xi \equiv \sigma$ on $U$ as $V$ is dense in $U$. Finally, $\xi$ is parallel on the whole $\mathbb{R}^2$ due to the fact that $U$ is a dense subset of $\mathbb{R}^2$. This prove the theorem under the hypothesis of item $(ii)$.
\end{proof}

\section{Extendability of projective vector fields and non-extendability of conformal ones}\label{sec.killing.conformal}

By recalling from the introduction that projective vector fields can be regarded as parallel sections of a suitable constructed linear connection, any such vector field defined on a open set $U\subset\mathbb{R}^2$ satisfying condition $\mathcal{R}^+$ of Definition \ref{conditionR+} can be extended to the whole $\mathbb{R}^2$. We underline that such result applies also to Killing, affine and homothetic vector fields as they are special projective vector fields.

\noindent
One can ask if this result holds for some more general class of vector fields, for instance for that of conformal ones. Below we see that, in dimension $2$, this is not the case. Indeed, in dimension $2$, conformal Killing vector fields cannot be regarded as parallel sections of a linear connection in a vector bundle, whereas, for dimension greater than $2$, they can be seen as parallel sections of the so called \emph{Geroch} connection \cite{Ge69,Ra06}.

\noindent
It is well-known that a conformal Killing vector field $X$ of the Euclidean plane $\mathbb{R}^2$ is given by a holomorphic function $f$. In fact, if
\[
X(x,y) = u(x,y) \frac{\partial}{\partial x} + v(x,y) \frac{\partial}{\partial y}
\]
is a conformal Killing vector field defined on $U$, then $f(z) = u(z) + \mathrm{i} v(z)$ belongs to the set of holomorphic functions $\mathcal{O}(U)$ on $U$. Indeed the flow $F^X_t$ consists of holomorphic maps, i.e., $F^X_t \in \mathcal{O}(U)$ for small values of $t$. Since the operators $\frac{\mathrm{d}}{\mathrm{d} t}$ and $\overline{\partial}$ commute, we have
\[  \left( \overline{\partial} \circ \frac{\mathrm{d}}{\mathrm{d} t} \right) F^X_t = \left(\frac{\mathrm{d}}{\mathrm{d} t} \circ \overline{\partial} \right) F^X_t = 0
\]
which shows that $f(z) = \frac{\mathrm{d}}{\mathrm{d} t}\big|_{t=0}  F^X_t(z)$ is holomorphic.

\smallskip\noindent
The function $1/z$ shows the existence of a conformal Killing vector field  $X$ defined on the open and connected subset $U = \mathbb{C}^* = \mathbb{C} \setminus 0$. Since $X$ is unbounded near $0$ (i.e. the Euclidean length of the vector field $X$ goes to infinity when approaching the origin), it cannot be extended to the whole plane $\mathbb{C}$. Observe that if a bounded conformal Killing vector field defined on an open set $U$ minus a discrete subset, then by Riemann's extension theorem it can be extended to the whole $U$ (see \cite{gianni_Killing} for a general discussion regarding arbitrary $2$-dimensional pseudo-Riemannian metrics).

\smallskip\noindent
Here we give an example of a \emph{bounded} conformal Killing vector $X$ field defined in the plane $\R^2$ minus a segment that cannot be extended to the whole plane $\R^2$. As explained in \cite[page 5]{Sp57}, in order to construct the Riemann surface of $w^2 = (z-r)\cdot(z-s)$, $r \neq s \in \mathbb{C}$, we cut $\mathbb{C} = \R^2$ along a segment $I$ connecting the branching points $r,s$ thus obtaining two single-valued branches, i.e. two holomorphic functions $w_1(z) \, , \, w_2(z)  : \mathbb{C} \setminus I \rightarrow \mathbb{C}$. Observe that both functions $w_1,w_2$ are bounded. Therefore, taking $f(z) = w_1(z)$, we get a bounded conformal Killing field which cannot be extended to the whole $\R^2$.

\begin{rem} \label{rem:SurSegI}The above example shows that Bryant's argument to solve the problem in MathOverflow can not be used if the domain of Killing vector field $X$ is the Riemann surface minus a segment. However, by using the Kostant connection and taking a radial extension of the parallel section associated to the Killing vector field $X$ we see that $X$ also extends in this case.
\end{rem}

\section{Extension of Killing vector fields of  $(\mathbb{R}^2,g)$ }\label{section:Kostant}

In this section we prove the following theorem.

\begin{thm}\label{thm:ExtKilling}
Let $M=(\mathbb{R}^2,g)$ be the plane endowed with a Riemann metric $g$. Let $\kappa$ be the Gaussian curvature of $g$. Assume that the differential $\mathrm{d }\kappa$ never vanish on $\mathbb{R}^2$.
Let $U \subset \mathbb{R}^2$ be a connected open and dense subset. If $(U,g)$ admits a Killing vector field $X$,
then it extends to a Killing vector field of $(\mathbb{R}^2,g)$.
\end{thm}

\noindent
For the proof of this theorem we will use the Kostant connection.

\subsection{Local description of the Kostant connection}

Let $(x,y)$ be local isothermal coordinates about a point of $(\mathbb{R}^2,g)$, i.e. the metric $g$ is given by
\[
ds^2 =  \lambda (dx^2 + dy^2) \, .
\]

\noindent
Let $J$ be the complex structure given by $J(\partial_x) = \partial_y$ and $J(\partial_y) = -\partial_x$. Recall that $J$ is parallel w.r.t. the Levi-Civita connection of $ds^2$.

\noindent
Consider the bundle $TM \oplus \mathfrak{so}(TM) $ endowed with the Kostant connection $\widetilde{\nabla}$ given by equation (\ref{eq.conn.Kostant}).
The sections $\xi_1 := (\partial_x,0)$, $\xi_2:=(\partial_y,0)$ and $\xi_3:=(0,J)$ are linearly independent, so that they form a frame of $TM \oplus \mathfrak{so}(TM)$. In order to prove Theorem \ref{thm:ExtKilling}, we need the following technical lemma.

\begin{lem}\label{lemma:rank2}
Let $R^K$ be the curvature tensor of $\widetilde{\nabla}$ and $\kappa$ the Gaussian curvature of $g$. The matrices of the operators $R^K_{\partial_x \partial_y} , (\widetilde{\nabla}_{\partial_x} R^K)_{\partial_x \partial_y}$ and  $(\widetilde{\nabla}_{\partial_y} R^K)_{\partial_x \partial_y}$ w.r.t. the frame $\{\xi_1, \xi_2, \xi_3\}$ are
\[ R^K_{\partial_x \partial_y} = \left(
                                   \begin{array}{ccc}
                                     0 & 0 & 0 \\
                                     0 & 0 & 0 \\
                                     -\kappa_x \lambda & -\kappa_y \lambda & 0 \\
                                   \end{array}
                                 \right)
\]

\[ (\widetilde{\nabla}_{\partial_x} R^K)_{\partial_x \partial_y} = \left(
                                   \begin{array}{ccc}
                                     0 & 0 & 0 \\
                                     \kappa_x \lambda & \kappa_y \lambda & 0 \\
                                     * & * & -\kappa_y \lambda \\
                                   \end{array}
                                 \right)
\]

\[ (\widetilde{\nabla}_{\partial_y} R^K)_{\partial_x \partial_y} = \left(
                                   \begin{array}{ccc}
                                     0 & 0 & 0 \\
                                     \kappa_x\lambda & \kappa_y \lambda & 0 \\
                                     *  & * & -\kappa_x \lambda \\
                                   \end{array}
                                 \right)
\]
\end{lem}
%

\begin{proof}
The proof of the above lemma is based on straightforward computations.
Let $\xi = (Z,0)$ be a section of the Kostant bundle. Then
\[ \widetilde{\nabla}_{\partial_x} \xi = (\nabla_{\partial_x} Z, -R(\partial_x, Z) )  \]
where $R(X,Y)Z = \kappa (X \wedge Y) Z$. Then
\[ \begin{aligned} \widetilde{\nabla}_{\partial_y} \widetilde{\nabla}_{\partial_x} \xi = (\nabla_{\partial_y} \nabla_{\partial_x} Z + R(\partial_x, Z)(\partial_y) , -\nabla_yR(\partial_x, Z) - R(\partial_y,\nabla_{\partial_x} Z))
\end{aligned}
\]
So
\[
\begin{aligned}
R^K_{\partial_x,\partial_y} \xi = (\nabla_{\partial_x} \nabla_{\partial_y} Z + R(\partial_y, Z)(\partial_x) , -\nabla_xR(\partial_y, Z) - R(\partial_x,\nabla_{\partial_y} Z) )\\
-(\nabla_{\partial_y} \nabla_{\partial_x} Z + R(\partial_x, Z)(\partial_y) , -\nabla_yR(\partial_x, Z) - R(\partial_y,\nabla_{\partial_x} Z) ) \\
= (0, \nabla_yR(\partial_x, Z) + R(\partial_y,\nabla_{\partial_x} Z) - \nabla_xR(\partial_y, Z) - R(\partial_x,\nabla_{\partial_y} Z))\\
(0,  (-\kappa \langle Z,\partial_y\rangle)_y J + \kappa \langle \nabla_{\partial_x}Z,\partial_x\rangle J - (\kappa \langle Z,\partial_x \rangle)_x J + \kappa \langle \nabla_{\partial_y}Z,\partial_y \rangle J )
\end{aligned}
\]
Then
\[
R^K_{\partial_x,\partial_y} \xi_1 = -\kappa_x \lambda \, \, \xi_3 \,, \quad R^K_{\partial_x,\partial_y} \xi_2 = -\kappa_y \lambda \, \, \xi_3 \,, \quad R^K_{\partial_x,\partial_y} \xi_3 = 0
\]

\vspace{0.5cm}
\noindent
Now we compute the covariant derivatives of the sections $\xi_1,\xi_2,\xi_3$.

\[
\begin{aligned}
\widetilde{\nabla}_{\partial_y} \xi_1 &= (\nabla_{\partial_y}\partial_x \, , \, -R(\partial_y,\partial_x))=
(\nabla_{\partial_y}\partial_x \, , \, \kappa. \partial_x \wedge \partial_y)
=(\nabla_{\partial_y}\partial_x \, , \, - \kappa \lambda J) =
\\
&= (\frac{\lambda_y}{2\lambda} \partial_x + \frac{\lambda_x}{2\lambda} \partial_y \, , \, - \kappa \lambda J)
= \frac{\lambda_y}{2\lambda}  \xi_1 + \frac{\lambda_x}{2\lambda} \xi_2  - \kappa \lambda \xi_3\\
\end{aligned}\]

\[\hspace{-2.2cm} \begin{aligned} \widetilde{\nabla}_{\partial_x} \xi_1 &= (\nabla_{\partial_x}\partial_x \, , \, 0) = (\frac{\lambda_x}{2\lambda} \partial_x - \frac{\lambda_y}{2\lambda} \partial_y, 0)
= \frac{\lambda_x}{2\lambda} \xi_1 -\frac{\lambda_y}{2\lambda}\xi_2
\end{aligned}\]

\[\hspace{-1.5cm} \begin{aligned} \widetilde{\nabla}_{\partial_y} \xi_2 &= (\nabla_{\partial_y}\partial_y \, , \, 0) = (-\frac{\lambda_x}{2\lambda} \partial_x + \frac{\lambda_y}{2\lambda} \partial_y, 0)=
 -\frac{\lambda_x}{2\lambda} \xi_1 + \frac{\lambda_y}{2\lambda}\xi_2
\end{aligned}\]

\[ \hspace{-3.5cm} \begin{aligned} \widetilde{\nabla}_{\partial_x} \xi_2 &= (\nabla_{\partial_x}\partial_y \, , \, -R(\partial_x,\partial_y))=
\frac{\lambda_y}{2\lambda}  \xi_1 + \frac{\lambda_x}{2\lambda} \xi_2  + \kappa \lambda \xi_3
\end{aligned}\]

\[\hspace{-2.8cm} \widetilde{\nabla}_{\partial_y} \xi_3 = (-J(\partial_y),0) = (\partial_x , 0) = \xi_1 \]
\[\hspace{-2.2cm} \widetilde{\nabla}_{\partial_x} \xi_3 = (-J(\partial_x),0) = (-\partial_y, 0) = -\xi_2 \]

\vspace{0.7cm}\noindent
Now we are in position to compute higher order covariant derivatives of $R^K$.

\[ \begin{aligned} (\widetilde{\nabla}_{\partial_x} R^K)_{\partial_x \partial_y} \xi_1 &= \widetilde{\nabla}_{\partial_x} (R^K_{\partial_x \partial_y} \xi_1 ) - R^K_{\nabla_{\partial_x}\partial_x \, \partial_y} \xi_1 - R^K_{{\partial_x} \, \nabla_{\partial_x}\partial_y}\xi_1 - R^K_{\partial_x \partial_y} \widetilde{\nabla}_{\partial_x} \xi_1\\
&= \widetilde{\nabla}_{\partial_x} (R^K_{\partial_x \partial_y} \xi_1 ) - R^K_{\frac{\lambda_x}{2 \lambda}\partial_x \, \partial_y} \xi_1 - R^K_{{\partial_x} \, \frac{\lambda_x}{2 \lambda}\partial_y}\xi_1- R^K_{\partial_x \partial_y} \widetilde{\nabla}_{\partial_x} \xi_1\\
&= \widetilde{\nabla}_{\partial_x} (R^K_{\partial_x \partial_y} \xi_1 ) - \frac{\lambda_x}{\lambda} R^K_{\partial_x \, \partial_y} \xi_1- R^K_{\partial_x \partial_y} \widetilde{\nabla}_{\partial_x} \xi_1\\
&= \widetilde{\nabla}_{\partial_x} (-\kappa_x \lambda \xi_3 ) - \frac{\lambda_x}{\lambda} R^K_{\partial_x \, \partial_y} \xi_1- R^K_{\partial_x \partial_y} \widetilde{\nabla}_{\partial_x} \xi_1\\
&= (-\kappa_x \lambda)_x \xi_3  - \kappa_x \lambda \widetilde{\nabla}_{\partial_x}\xi_3  - \frac{\lambda_x}{\lambda} R^K_{\partial_x \, \partial_y} \xi_1- R^K_{\partial_x \partial_y} \widetilde{\nabla}_{\partial_x} \xi_1\\
&= (-\kappa_x \lambda)_x \xi_3  - \kappa_x \lambda \widetilde{\nabla}_{\partial_x}\xi_3  + \lambda_x \kappa_x \xi_3- R^K_{\partial_x \partial_y} \widetilde{\nabla}_{\partial_x} \xi_1\\
&= (-\kappa_x \lambda)_x \xi_3  + \kappa_x \lambda \xi_2  + \lambda_x \kappa_x \xi_3- R^K_{\partial_x \partial_y} \widetilde{\nabla}_{\partial_x} \xi_1\\
&=  \kappa_x \lambda \xi_2  - \kappa_{xx} \lambda \xi_3- R^K_{\partial_x \partial_y} \widetilde{\nabla}_{\partial_x} \xi_1\\
&=  \kappa_x \lambda \xi_2  - \kappa_{xx} \lambda \xi_3- R^K_{\partial_x \partial_y}(\frac{\lambda_x}{2\lambda} \xi_1 - \frac{\lambda_y}{2\lambda}\xi_2)\\
&=  \kappa_x \lambda \xi_2  - \kappa_{xx} \lambda \xi_3+\frac{\lambda_x}{2\lambda} \kappa_x \lambda \xi_3 - \kappa_y \lambda \frac{\lambda_y}{2\lambda}\xi_3\\
&=  \kappa_x \lambda \xi_2  + \left( \frac{\lambda_x \kappa_x - \kappa_y \lambda_y}{2}- k_{xx}\lambda \right) \xi_3\\
\end{aligned} \]

\[ \begin{aligned} (\widetilde{\nabla}_{\partial_x} R^K)_{\partial_x \partial_y} \xi_2 &= \widetilde{\nabla}_{\partial_x} (R^K_{\partial_x \partial_y} \xi_2 ) - R^K_{\nabla_{\partial_x}\partial_x \, \partial_y} \xi_2 - R^K_{{\partial_x} \, \nabla_{\partial_x}\partial_y}\xi_2 - R^K_{\partial_x \partial_y} \widetilde{\nabla}_{\partial_x} \xi_2\\
&= \widetilde{\nabla}_{\partial_x} (R^K_{\partial_x \partial_y} \xi_2 ) - R^K_{\frac{\lambda_x}{2 \lambda}\partial_x \, \partial_y} \xi_2 - R^K_{{\partial_x} \, \frac{\lambda_x}{2 \lambda}\partial_y}\xi_2 - R^K_{\partial_x \partial_y} \widetilde{\nabla}_{\partial_x} \xi_2\\
&= \widetilde{\nabla}_{\partial_x} (R^K_{\partial_x \partial_y} \xi_2 ) - \frac{\lambda_x}{\lambda} R^K_{\partial_x \, \partial_y} \xi_2 - R^K_{\partial_x \partial_y} \widetilde{\nabla}_{\partial_x} \xi_2\\
&= \widetilde{\nabla}_{\partial_x} (-\kappa_y \lambda \xi_3 ) - \frac{\lambda_x}{\lambda} R^K_{\partial_x \, \partial_y} \xi_2- R^K_{\partial_x \partial_y} \widetilde{\nabla}_{\partial_x} \xi_2\\
&= (-\kappa_y \lambda)_x \xi_3  - \kappa_y \lambda \widetilde{\nabla}_{\partial_x}\xi_3  - \frac{\lambda_x}{\lambda} R^K_{\partial_x \, \partial_y} \xi_2- R^K_{\partial_x \partial_y} \widetilde{\nabla}_{\partial_x} \xi_2\\
&= (-\kappa_y \lambda)_x \xi_3  - \kappa_y \lambda \widetilde{\nabla}_{\partial_x}\xi_3  + \lambda_x \kappa_y \xi_3- R^K_{\partial_x \partial_y} \widetilde{\nabla}_{\partial_x} \xi_2\\
&= (-\kappa_y \lambda)_x \xi_3  + \kappa_y \lambda \xi_2  + \lambda_x \kappa_y \xi_3- R^K_{\partial_x \partial_y} \widetilde{\nabla}_{\partial_x} \xi_2\\
&=  \kappa_y \lambda \xi_2  - \kappa_{xy} \lambda \xi_3- R^K_{\partial_x \partial_y} \widetilde{\nabla}_{\partial_x} \xi_2\\
&=  \kappa_y \lambda \xi_2  - \kappa_{xy} \lambda \xi_3+ R^K_{\partial_x \partial_y}\left( \frac{\lambda_y}{2\lambda}  \xi_1 + \frac{\lambda_x}{2\lambda} \xi_2  + \kappa \lambda \xi_3\right)\\
&=  \kappa_y \lambda \xi_2  - \kappa_{xy} \lambda \xi_3+ R^K_{\partial_x \partial_y}\left( \frac{\lambda_y}{2\lambda}  \xi_1 + \frac{\lambda_x}{2\lambda} \xi_2 \right)\\
&=  \kappa_y \lambda \xi_2  - \kappa_{xy} \lambda \xi_3- \frac{\lambda_y \kappa_x}{2}  \xi_3 - \frac{\lambda_x\kappa_y}{2} \xi_3 \\
&=  \kappa_y \lambda \xi_2  - \left( \kappa_{xy} \lambda + \frac{\lambda_y \kappa_x}{2} + \frac{\lambda_x\kappa_y}{2} \right) \xi_3 \\
\end{aligned} \]

\[ \begin{aligned} (\widetilde{\nabla}_{\partial_x} R^K)_{\partial_x \partial_y} \xi_3 &= \widetilde{\nabla}_{\partial_x} (R^K_{\partial_x \partial_y} \xi_3 ) - R^K_{\nabla_{\partial_x}\partial_x \, \partial_y} \xi_3 - R^K_{{\partial_x} \, \nabla_{\partial_x}\partial_y}\xi_3 - R^K_{\partial_x \partial_y} \widetilde{\nabla}_{\partial_x} \xi_3\\
&=- R^K_{\partial_x \partial_y} \widetilde{\nabla}_{\partial_x} \xi_3
= R^K_{\partial_x \partial_y} \xi_2
= -\kappa_y \lambda \xi_3
\end{aligned} \]
The lemma follows by taking into account the above computations.
\end{proof}

\subsection{Proof of Theorem \ref{thm:ExtKilling}}

\begin{proof}
Since the domain $U$ of the Killing vector field $X$ is assumed to be dense it is enough to show that $X$ can be extended about any point of the boundary of $U$.
If the Killing vector field is zero then the lemma is trivial. So we will assume that $X$ is not zero. Recall that this implies that the zero set of $X$ is discrete.
Let $p_0$ be a boundary point of $U$.
Then Lemma \ref{lemma:rank2} implies that either $(\widetilde{\nabla}_{\partial_x} R^K)_{\partial_x \partial_y}$ or $(\widetilde{\nabla}_{\partial_y} R^K)_{\partial_x \partial_y}$ has rank 2 in a small disk $B_{p_0}$ of $p_0$.
Assume that  $(\widetilde{\nabla}_{\partial_x} R^K)_{\partial_x \partial_y}$ has rank 2 on $B_{p_0}$.
Then the kernel of $(\widetilde{\nabla}_{\partial_x} R^K)_{\partial_x \partial_y}$ defines a smooth real line bundle $\mathcal{L}$ of the restriction to $B_p$ of the Kostant bundle  $TM \oplus \mathfrak{so}(TM)$.

\smallskip\noindent
We claim that $\mathcal{L}$ is a flat parallel line bundle w.r.t. the Kostant connection $\widetilde{\nabla}$.

\medskip\noindent
In fact, let $\xi$ be a generator of $\mathcal{L}$ on $B_{p_0}$ and $Y$ any vector field of $B_{p_0}$. First we show that \begin{equation}\label{eq.xi.nabla.tilde}
\xi \wedge \widetilde{\nabla}_Y \xi \equiv 0.
\end{equation}
Observe that, on the intersection $B_{p_0} \cap U$, the parallel section $\sigma$ induced by the Killing vector field $X$ must take values in $\mathcal{L}$. Since the zero set of $X$ is discrete, we get that equality \eqref{eq.xi.nabla.tilde} holds on  $B_{p_0} \cap U$ hence it holds on $B_{p_0}$ in view of the fact we assume $U$ to be dense.

\smallskip\noindent
This shows that any covariant derivative of the generator $\xi$ is in $\mathcal{L}$, so $\mathcal{L}$ is $\widetilde{\nabla}$-parallel.
Thus $\mathcal{L}$ is flat since $\sigma$ is a parallel section taking values on $\mathcal{L}|_{B_{p_0} \cap U}$ with $U$ a dense subset. Then the section $\sigma$ can be extended to a parallel section of $\mathcal{L}$ on the whole $B_{p_0}$ because $B_{p_0}$  is simply connected.
This shows that $X$ extends to a Killing vector field of $G:= U \cup B_{p_0}$.
\end{proof}

\subsection*{Acknowledgments}
The authors thank D. Alekseevsky, S. Fornaro, T. Kirschner, V. Matveev, C. Olmos, P. Tilli and F. Vittone for useful suggestions and discussions.

\end{document}